\documentclass[12pt,a4paper,twoside,english]{smfart}

\usepackage{lmodern,mdframed,xcolor}
\usepackage[utf8]{inputenc}
\usepackage[francais]{babel}

\usepackage[OT2,T1]{fontenc}

\mdfdefinestyle{MyFrame}{%
    linecolor=blue,
    innerrightmargin=5pt,
    innerleftmargin=5pt,
    splittopskip=\topskip,
    backgroundcolor=yellow!60}

\usepackage{amssymb}
\usepackage{amsmath}
\usepackage{smfthm,mathabx}
\usepackage{enumerate}
\usepackage{amsthm}
\usepackage{amsfonts}
\usepackage{graphicx}
\usepackage[colorlinks=true,linkcolor=black,citecolor=black,urlcolor=black]{hyperref}
\usepackage{fancyhdr}
\usepackage[all,cmtip]{xy}
\usepackage{alltt}
\usepackage{footmisc}
\usepackage{smfthm}

\usepackage{calrsfs}

\usepackage{fullpage}

\xyoption{all}

\xyoption{frame}

\newcommand{\finpreuve}{\mbox{} \hfill \mbox{$\Box$}}

\def\Q{{\mathbb Q}}
\def\Z{{\mathbb Z}}
\def\fq{{\mathbb F}}

\def\p{{\mathfrak p}}

\def\A{{\mathfrak A}}

\def\R{{\rm R}}

\def\G{{\rm G}}
\def\K{{\rm K}}
\def\L{{\rm L}}
\def\F{{\rm F}}

\def\V{{\rm V}}

\def\O{{\mathcal O}}

\def\H{{\rm H}}

\def\Gal{{\rm Gal}}

\def\sl{{\mathfrak s} {\mathfrak l}}
\def\1{{\bf 1}}

\newtheorem*{Theorem}{Theorem}

\author{Farshid Hajir, Christian Maire, Ravi Ramakrishna}
\address{Department of Mathematics, University of Massachussetts, Amherst, MA 01003, USA}
 \address{FEMTO-ST Institute, Universit\'e Bourgogne Franche-Comt\'e, CNRS,  15B avenue des Montboucons, 25000 Besan\c con, FRANCE} 
\address{Department of Mathematics, Cornell University, Ithaca, USA}
\email{hajir@math.umass.edu, christian.maire@univ-fcomte.fr, ravi@math.cornell.edu}

\begin{document}

\date{\today}

\title{Infinite  class field towers of number fields of prime power discriminant}

\begin{abstract}  For every prime number $p$,  we show  the existence of a solvable number field $\L$ ramified only at $\{p, \infty\}$ whose $p$-Hilbert Class field tower is
infinite.
\end{abstract}

\thanks{We all thank Mathematisches Forschungsinstitut Oberwolfach for sponsoring a ``Research in Pairs'' stay during which this work was done. 
The second author  was partially supported by the ANR project FLAIR (ANR-17-CE40-0012) and  by the EIPHI Graduate School (ANR-17-EURE-0002). The third author was supported by Simons collaboration grant 524863.}

\maketitle

\subjclass{ }

For a number field $\rm L$ of degree $n$ over $\Q$, the root discriminant is defined to be $D_{\rm L}^{1/n}$ where $D_{\rm L}$ is the absolute value of the discriminant of $\rm L$. Given a finite set $S$ of places of $\Q$, it is an old question as to whether there is an infinite sequence of number fields unramified outside $S$ with bounded root discriminant.  This question is related to the constants of Martinet \cite{Martinet} and Odlyzko's bounds \cite{Odlyzko}. Since the root discriminant is constant in unramified extensions, an approach to answering the previous question in the positive is to find a number field $\rm L$ (of finite degree) unramified outside $S$ having an infinite class field tower. 
In the  case of $\K/\Q$ quadratic, it is a classical result of Golod and Shafarevich that if  $\K/\Q$ is ramified at at least $ 8$ places, then $\K$ has
an infinite class field tower. On the other hand, if $p$ is a prime, and $S=\{p,\infty\}$, the question becomes whether there exist number fields with $p$-power discriminant having an infinite unramified extension.
Schmitals \cite{Schmithals} and  Schoof \cite{Schoof}
 produced a few isolated examples of this type.  
 See also \cite{Hajir-Maire-JA}, \cite{Leshin}, etc.
For $p\in \{2,3,5\}$, Hoelscher \cite{Hoelscher} announced the existence of number fields unramified outside $\{ p, \infty\}$  and having an infinite Hilbert class field tower.
Here  we prove:

\begin{Theorem}\label{Theorem:main}
For every prime number $p$, there exists a solvable extension  $\rm L/\Q$, ramified only at $\{p,\infty\}$, having an infinite Hilbert $p$-class field tower. Consequently, there exists an infinite nested sequence of number fields of $p$-power discriminant with bounded root discriminant.
\end{Theorem}
Our proof is based on the idea of cutting of wild towers introduced in \cite{HMR}; in particular it does not involve the usual technique  of genus theory. For the more refined question where $S$ consists of a single prime number $p$ (i.e. if we focus our attention on totally real fields only), we do not know whether for every prime $p$, there is a totally real number field of $p$-power discriminant having an infinite Hilbert class field tower. In \cite[Corollary 4.4]{Schoof} it is shown that $\Q(\sqrt{39345017})$ (which is ramified only at the prime
$39345017$) has infinite Hilbert class field tower. In \cite{Shanks}, Shanks  studied primes of the form $p=a^2+3a+9$ and the corresponding totally real
cubic subfields $\K \subset \Q(\mu_p)$ and
 showed the minimal polynomials of $\K$ are $x^3-ax^2-(a+3)x-1$. 
Taking $a=17279$ so
 $p=298615687$, one can compute that  the $2$-part of the class group of $\K$  has rank $6$. It is not hard to see, using the Golod-Shafarevich criterion, that $\K$ has infinite
$2$-Hilbert class field tower.
Thus some examples exist in the totally real case.

\section{The results we need}

Let $p$ be a prime number. Let $\K/\Q$ be a finite  Galois extension. Assume $\mu_p \subset \K$ and moreover that $\K$ is totally imaginary when $p=2$.  
 For a prime $\p$ of $\K$ dividing  $p$ denote by $e$ (resp. $f$) the ramification index (resp. the residue degree) of $\p$ in $\K/\Q$.

\subsection{On the group $\G_S$}
Denote by $S$ the set of places of $\K$ above $p$, and consider $\K_S$ the maximal pro-$p$ extension of $\K$ unramified outside~$S$; put $\G_S=\Gal(\K_S/\K)$. Let $g=|S|$ be the number of places of $\K$ above $p$. 

Let $h_\K'$ be the $S$-class number of $\K$. By class field theory, $h_\K'$ is equal to $[\K':\K]$ where $\K'/\K$ is the maximal abelian of $\K$ unramified everywhere  in which all places  of $S$ split completely. The Kummer radical of the $p$-elementary subextension  $\K'(p)/\K$ of $\K'/\K$ is 
$$\V_S:=\{x\in \K^\times \mid x{\mathcal O}_\K=\A^p, \ x \in \K_v^{\times p}, \forall v \in S\}.$$ In particular  $p\nmid h_\K'$ if and only if  $\V_S/\K^{\times p}$ is trivial.

\medskip

By work of Koch and Shafarevich  the pro-$p$ group $\G_S$ is finitely presented. More precisely, in our situation  one has:

\begin{Theorem} \label{Theorem:rappel} Let $\K/\Q$ be a totally imaginary Galois extension containing $\mu_p$. Let $S=\{p,\infty\}$.  Then
$$\dim H^1(\G_S,\fq_p)=\frac{efg}{2}+1 + \dim  H^2(\G_S,\fq_p) $$
and $$\dim H^2(\G_S,\fq_p) = g-1 +  \dim \V_S/\K^{\times p}.$$
\end{Theorem}

\begin{proof} This is well-known, see for example \cite[Corollary 8.7.5 and Theorem 10.7.3]{NSW}.
\end{proof}

We immediately have:

\begin{coro} \label{coro:main}
If $p\nmid h_\K$ then $\dim H^1(\G_S,\fq_p)= g(\frac{ef}{2}+1)$ and $\dim H^2(\G_S,\fq_p)=g-1 $.
\end{coro}

\subsection{The cutting towers strategy}

\subsubsection{The Golod-Shafarevich Theorem}

Let $\G$ be a finitely generated pro-$p$ group. Consider a minimal presentation  $1\rightarrow \R \rightarrow \F\stackrel{\varphi}{\rightarrow} \G$ of $\G$, where $\F$ is a  free pro-$p$ group. Set $d=d(\G)=d(\F)$, the number of generators of $\G$ and $\F$.
 Suppose that $\R=\langle \rho_1,\cdots, \rho_r \rangle^{Norm}$ is generated   as normal subgroup of $\F$ by a finite set of relations  $\rho_i$.

We recall the depth function $\omega$ on $\F$. See \cite[Appendice]{Lazard} or \cite{Koch}  for more details.
The augmentation ideal $I$  of $\fq_p [[\G ]]$ is, by definition,  
generated by the set of elements $\{g-e\}_{g\in \G}$.
Then for $e\neq g\in \F$, define 
$\omega(g)=\max_k \{ g-e \in I^k\}$; put $\omega(e)= \infty$. It is not difficult to see that $\omega([g,g']) \geq 2$ and that $\omega(g^{p^k}) \geq p^k$ for every $g,g' \in \G$ and $k\in \Z_{>0}$.
Observe also that as the presentation $\varphi$ is minimal,  $\omega(\rho_i) \geq 2$ for all the relations $\rho_i$.

\medskip

The  Golod-Shafarevich polynomial associated to the presentation $\varphi$ of $\G$ is the polynomial $P_\G(t)=1-dt + \sum_i t^{\omega(\rho_i)}$.

\begin{Theorem}[Golod-Shafarevich, Vinberg \cite{Vinberg}]
If $\G$ is finite then $P_\G(t) >0$ for all $t\in ]0,1[$.
\end{Theorem}
Of course if we have no information about the $\rho_i$'s we may take $1-dt +rt^2$ 
(where $r=\dim H^2(G,\fq_p)$)
as Golod-Shafarevich polynomial for $\G$: if  $1-dt+rt^2$ is negative
at $t_0 \in ]0,1[$, then  $P_\G(t_0) < 0$ and $\G$ is infinite.

\medskip
 
We can also define a depth function $\omega_\G$ on $\G$ associated to its augmentation ideal. Then:
 
 \begin{prop} For every $g \in \G$, one has $$\omega_\G(g)=\max\{ \omega(y), \varphi(y)=g\}.$$ 
 \end{prop}

\begin{proof}
See \cite[Appendice 3, Theorem 3.5]{Lazard}.
\end{proof}

\medskip

We now study quotients  $\Gamma$ of $\G$ such that $d(\G)=d(\Gamma)$. In this case, the initial minimal presentation of $\G$ induces a minimal presentation of $\Gamma$
$$\xymatrix{1  \ar[r] & \R \ar[r] & \F \ar[r]^\varphi \ar@{.>>}[rd] & \G \ar[r] \ar@{->>}[d] & 1 \\
 & & & \Gamma &}$$
 
 Suppose that $\Gamma=\G/\langle x_1,\cdots, x_m\rangle^{Norm}$.
 Here $\langle x_1,\cdots, x_m\rangle^{Norm}$ is the normal subgroup of $\G$ generated by the $x_i$'s. Lift the $x_i$'s to  $y_i \in \F$ such that $\omega_\G(x_i)=\omega(y_i)$ for each $i$.
 Hence, $\Gamma=\F/\R'$, where $\R'=\R\langle y_1, \cdots, y_m\rangle^{Norm}$.
 In particular, if $\R=\langle \rho_1,\cdots, \rho_r \rangle^{Norm}$, then  $\R'=\langle \rho_1, \cdots, \rho_r, y_1, \cdots, y_m\rangle^{Norm}$.
 
 If we have no information about the $\rho_i$'s,  we can take $P_\Gamma(t)=1-dt+rt^2+\sum_i t^{\omega(y_i)}$ as Golod-Shafarevich polynomial for $\Gamma$.

\subsubsection{Cutting of $\G_S$} \label{section:cutting}

We want to consider some special quotients $\Gamma$ of $\G_S$, this is what we call ``cutting wild towers''. 
 
 \medskip
 Each place $v\in S$ corresponds to some extension  $\K_v/\Q_p$ (in fact these fields are isomorphic as $\K/\Q$ is Galois) of degree $ef$.
 Then, as $\mu_p \subset \K_v$,  the ${\mathbb F}_p$-vector space $\K^\times_v/\K^{\times p}_v$ has dimension $ef+2$,
and local class field theory implies the Galois group 
of the maximal pro-$p$ extension of $\K_v$ is generated by $ef+2$ elements.  
 Thus the decomposition subgroup $\G_v$ of $v$ in $\K_S/\K$ is 
 generated by at most $ef+2$ elements $z_{i,v}$. Consider now the  commutators $[z_{i,v},z_{k,v}] $ of all these elements; there are at most $\binom{ef+2}{2}$ such elements. Now we cut $\G_S$ by  $\langle [z_{i,v},z_{k,v}], i,k; v\in S\rangle^{Norm}$, and denote by $\Gamma$ the corresponding quotient.  As $\omega_{\G_S}([z_{i,v},z_{k,v}]) \geq 2$, one can take $P_\Gamma=1-dt+rt^2+g\binom{ef+2}{2} t^2$ as Golod-Shafarevich polynomial for $\Gamma$; here $d=\dim H^1(\G_S,\fq_p)$ and $r=\dim H^2(\G_S,\fq_p)$.
 This quotient  $\Gamma$ of $\G_S$  corresponds to the maximal subextension $\K_S^{loc-ab}/\K$ of $\K_S/\K$ locally abelian everywhere.
Observe that $\K_S^{loc-ab}/\K$ contains the compositum of all $\Z_p$-extensions. 

\medskip

Suppose that there exists some $t_0 \in ]0,1[$ such that $P_\Gamma(t_0)<0$.
We will then cut the infinite pro-$p$ group $\Gamma$ by all the  $z_{v,i}^{p^k}$ for some large $k$. 
There are $g(ef+2)$ such elements.
Denote by $\Gamma_k$ the new quotient and by $\K_S^{[k]}$ the new extension of $\K$ corresponding to $\Gamma_k$. 
Since $\omega_\Gamma(z_{v,i}^{p^k})\geq p^k$,
we may take $P_{\Gamma_k}(t)=P_\Gamma(t) + g(ef+2)t^{p^k}$ as the Golod-Shafarevich polynomial for $\Gamma_k$.
When $k$ is sufficiently large, clearly $P_{\Gamma}(t_0) < 0 \implies
P_{\Gamma_k}(t_0)<0$, so  $\K_S^{[k]}/\K$ is infinite.  

The main interest of $\K_S^{[k]}/\K$ is:

\begin{prop}\label{prop:basechange} Suppose $\K_S^{[k]}/\K$ infinite. Then 
there exists a finite  subextension $\rm L/\K$ of $\K_S^{[k]}/\K$ having an infinite Hilbert $p$-class field tower.
\end{prop}

\begin{proof} In   $\K_S^{[k]}/\K$ the (wild) ramification is finite: indeed  for each $v\in S$, the decomposition groups in  $\K_S^{[k]}/\K$ are abelian, finitely generated  and of finite exponent.
 There exists a finite extension $\rm L/\K$ inside $\K_S^{[k]}/\K$ absorbing all the ramification, so
 $\K_S^{[k]}/\rm L$ is  unramified everywhere and infinite.
\end{proof}

\section{Proof}

\begin{prop}\label{prop:0} Let $\K/\Q$ be finite Galois with $\mu_p \subset \K$.  
Assume that 
$g\geq 8$.
Then there exists a finite subextension $\rm L/\K$ of $\K_S/\K$ such that the Hilbert $p$-class field tower of $\rm L$ is infinite.
\end{prop}
\begin{proof} 
Let $\H$ be the ``top'' of the Hilbert Class Field Tower of $\K$. If $\H/\K$ is infinite, we are done,
so suppose $[\H:\K] < \infty$. Note that $\H$ has class number $1$ so by Corollary \ref{coro:main}, working over $\H$,
$\dim H^1(\G_S,\fq_p)=g\left(\frac{ef}{2}+1\right)$ and $\dim H^1(\G_S,\fq_p)=g-1$.
\smallskip
As  in Section \ref{section:cutting},  consider  the quotient $\Gamma$ of $\G_S$ by the normal subgroup generated by the local commutators at each $v\in S$; one has  $\binom{ef+2}{2}$ such commutators. The group $\Gamma$ can be described by $d:=g\left(\frac{ef}{2}+1\right)$ generators and by $r:=g-1+ g\frac{(ef+2)(ef+1)}{2}$ relations.
\smallskip

The Golod-Shafarevich polynomial of $\Gamma$ may be written as  $P_{\Gamma}(t)=1-d t+r t^2$, 
when assuming the worst case that all the relations are of depth $2$.  Clearly $d/2r <1$, 
and $P_{\Gamma}(d/2r)=1-\frac{d^2}{4r}$. In particular, if $P_{\Gamma}(d/2r)<0$, then one has room to cut by some large 
$p$-power of the local generators, in order to obtain at the end  some finite local groups.  For the result to follow, we thus need $4r < d^2$, or equivalently
$$4 \left(g-1+g\frac{(ef+2)(ef+1)}{2}\right) \stackrel{?}{< }\frac{g^2}4 (ef+2)^2$$
which is equivalent to 
$$16(g-1) +8g(ef+2)(ef+1) \stackrel{?}{<} g^2( ef+2)^2.$$
Replacing the $16(g-1)$ term on the left by $16g$ and dividing by $g$, and setting $x=ef$, we need
to verify
$$16+8(x+2)(x+1) \stackrel{?}{<}g(x+2)^2.$$
This holds for $g \geq 8$ and $x=ef \geq 1$.
 Proposition \ref{prop:basechange} allows us to conclude $\K_S^{[k]}/\K$ is infinite when $k$ is  sufficiently large.  
\end{proof}
\medskip
\noindent
{\em Proof Theorem~\ref{Theorem:main} }:
Recall that the principal prime $\p = (1-\zeta_{p^s})$ of $\Q(\zeta_{p^s})$ is the unique prime dividing $p$ and  by class field theory $\p$ splits  completely in the Hilbert
class field $\H$ of $\Q(\zeta_{p^s})$.  Thus if the class group has order at least $8$,  Proposition~\ref{prop:0}  applied to the solvable number field $\H$ gives the result.

\medskip

 In the proof of  \cite[Corollary 11.17]{Washington}, the class number of $\Q(\zeta_{p^r})$ is shown to be at least $10^9$ for  $\phi(p^r)=p^{r-1}(p-1) >220$.
Choosing $r\geq 9$  for {\it any} $p$ completes the proof of the Theorem. 
   
\medskip
A slightly more detailed analysis using Table \S 3   of \cite{Washington} shows the fields below suffice:
 \vskip1em

$\begin{array}{cll}
p & \K & g=h\\
p>23 & \K=\Q(\zeta_p) &  \geq 8\\
7 \leq p \leq 23 & \K=\Q(\zeta_{p^2}) & \geq 43\\
p=5 & \K=\Q(\zeta_{125}) & 57708445601\\
p=3 & \K=\Q(\zeta_{81}) & 2593 \\
p=2 & \K=\Q(\zeta_{64}) & 17
\end{array}$

\hfill   $\square$

\begin{rema}
In \cite{Hoelscher}  a proof of the Theorem for $p=2,3$ and $5$ was given. Our proof is partially modeled on the ideas there, namely considering the Hilbert class
field of a cyclotomic field. There are two cases in \cite{Hoelscher}: Case I, where the Hilbert class field tower is infinite; and Case II, where ramification is allowed
at one prime above
$p$ in the Hilbert class field $\H$ and  a $\Z/p$-extension of $\H$ ramified at exactly this prime is used. Gras has given a criterion for such an extension to exist: see 
\cite[Chapter V, Corollary 2.4.4]{gras}. Gras' criterion is not verified in \cite{Hoelscher}. Given the size of the number fields $\H$, it seems very difficult to do so.
We therefore we regard the results of \cite{Hoelscher} as incomplete.


\end{rema}


\end{document}